\documentclass
[
    a4paper,
    DIV=10,
    abstracton
]
{scrartcl}

\usepackage{
    amsmath,
    amssymb,
    amsthm,
    bm,
    graphicx,
    caption,
    subcaption,
    enumerate,
    xcolor,
    mathtools,
    authblk,
    dsfont,
    todonotes,
    bm,
    bbm
}

\usepackage[utf8]{inputenc}
\usepackage[pdffitwindow=false,
            plainpages=false,
            pdfpagelabels=true,
            pdfpagemode=UseOutlines,
            pdfpagelayout=SinglePage,
            bookmarks=false,
            colorlinks=true,
            hyperfootnotes=false,
            linkcolor=blue,
            urlcolor=blue!30!black,
            citecolor=green!50!black]{hyperref}

\newcommand{\R}{\mathbb{R}}                                     
\newcommand{\X}{\mathbb{X}}                                     
\newcommand{\Y}{\mathbb{Y}}

\newcommand{\ts}{\hspace*{0.1em}} 

\newcommand{\hlC}[2][]{{#2}}
\newcommand{\hl}[1]{{#1}}

\newcommand{\M}{\mathbb{M}}                                     
\newcommand{\Q}{{\mathcal{Q}}}

\mathtoolsset{centercolon} 
\newtheorem{theorem}{Theorem}[section]
\newtheorem{corollary}[theorem]{Corollary}

\newtheorem{definition}[theorem]{Definition}
\theoremstyle{definition}
\newtheorem{example}[theorem]{Example}
\newtheorem{remark}[theorem]{Remark}

\DeclareMathOperator*{\argmin}{arg\,min}

\title{A weak characterization of slow variables in stochastic dynamical systems}

\author[1]{Andreas Bittracher}
\author[1,2]{Christof Sch\"utte}

\affil[1]{Department of Mathematics and Computer Science, Freie Universit\"at Berlin, Germany}
\affil[2]{Zuse Institute Berlin, Germany}
\date{}

\begin{document}
\maketitle

\begin{abstract}
	We present a novel characterization of slow variables for continuous Markov processes that provably preserve the slow timescales.
	These slow variables are known as reaction coordinates in molecular dynamical applications, where they play a key role in system analysis and coarse graining.
	The defining characteristics of these slow variables is that they parametrize a so-called transition manifold, a low-dimensional manifold in a certain density function space that emerges with progressive equilibration of the system's fast variables.
	The existence of said manifold was previously predicted for certain classes of metastable and slow-fast systems. However, in the original work, the existence of the manifold hinges on the pointwise convergence of the system's transition density functions towards it. We show in this work that a convergence in average with respect to the system's stationary measure is sufficient to yield reaction coordinates with the same key qualities.
	This allows one to accurately predict the timescale preservation in systems where the old theory is not applicable or would give overly pessimistic results.
	Moreover, the new characterization is still constructive, in that it allows for the algorithmic identification of a good slow variable.
	The improved characterization, the error prediction and the variable construction are demonstrated by a small metastable system.
	
\end{abstract}

\section{Introduction}

The ability and practice to perform all-atom molecular simulations of more and more complex biochemical systems has led to an unprecedented increase in the available amount of dynamical data about those systems. This has exponentiated the importance to identify good chemical reaction coordinates (RCs), low-dimensional observables of the full system that are associated with the relevant, often slowly-progressing sub-processes. For one, a meaningful RC permits insight into the essential mechanisms and parameters of a reaction, by acting as a filter for the overwhelming complexity of the data. As an example, computing the free energy (also known as the potential of mean force) along such a coordinate is typically used for identifying energy barriers and associated transition states~\cite{Smith1999,Daldrop18}. 
RCs are also essential for the development of accurate reduced dynamical models. The Mori-Zwanzig formalism and related schemes~\cite{Zwa61,PaSt08,ZHS16,Netz18} can be used to derive approximate closed equations of motion of the dynamics projected onto the image space of the RC. Depending on the chosen RC, the essential dynamical properties of the reduced model --- such as transition rates between reactant and product --- may or may not resemble those of the original system~\cite{ZS17}. Finally, accelerated sampling schemes such as metadynamics~\cite{metadynamics-review}, Blue Moon sampling~\cite{CKV05} and umbrella sampling~\cite{Torrie1977} also rely heavily on an accurate RC to guide them efficiently into unexplored territory.

In each of those applications, the result depends crucially on the ``quality'' of the RC, an elusive measure for how well \hlC[C1]{the RC suits the specified task. In most cases, this quality can be brought down to how well the RC ``captures the essential dynamics'', in particular the rates of transitions between reactant and product state  (see also~\cite{Baron16} for an in-depth review on the effect of poorly chosen RCs on different classic rate theories). Due to this ambiguity, the search for universal and mathematically rigorous optimality criteria for RCs remains an active field of research, and numerous new approaches have been suggested during the last decade.} For reactions involving one clearly defined reactant and product state, a in multiple ways ideal RC is the committor function~\cite{Dinner,Best2005}, a one-dimensional observable that in each point describes the probability to hit the product state before returning to the reactant state. As the committor function is notoriously hard to compute, advanced numerical schemes have been developed to either approximate it efficiently~\cite{ElEtAl17}, or find RCs that are equivalent by certain metrics~\cite{Peters06}. Still, the computation of committor-like RCs often remains out of reach for high-dimensional systems.

For systems where the relevant behavior involves transitions between more than two states~\cite{Sirur16}, where the reaction is not adequately described by a transition between isolated states~\cite{Sengupta19}, or where the states are not known or cannot be computed, other optimality criteria must be employed. Here one common approach is to demand the preservation of the system's  longest (equilibration) time scales under projection of the dynamics onto the RC. This leads naturally to a characterization of RCs in terms of the eigenvalues of the system's transfer operator, a widely used mathematical tool for time scale analysis in molecular dynamics and beyond~\cite{KNKWKSN18,deju:99,Schuette1999,CN14,Williams15}.
It is in this setting where the authors and coworkers have previously proposed a novel mathematical framework for the characterization and numerical computation of ideal RCs~\cite{Bittracher2017}. The proposed theory builds on the insight that in many systems, the equilibration of the fast sub-processes over time manifests as the convergence of the system's transition density functions towards a certain low-dimensional manifold in density space, the so-called transition manifold (TM). This convergence is observed even if there is no equivalent low-dimensional structure in state space, such as a transition pathway between isolated states. Any parametrization of the TM then can in theory be used to construct an ideal RC.

The framework demands that the convergence towards the TM must occur for \emph{all} transition density functions, i.e., for every conceivable starting state. In practice however, this rather strong condition is often violated for starting states with high potential energy, as the associated transition density functions may stay far away from any sensible candidate TM for all times. The probability to encounter these states in the canonical ensemble is however exponentially low, and thus should not contribute significantly to the shape of the RC. Indeed, the numerical methods built around parametrizing the TM are able to successfully deal with this problem by heuristically ignoring sparse outliers by tuning the manifold learning algorithm~\cite{BBS18,BKHKS19}.

Still, a rigorous argument for why those outliers can be safely ignored was lacking so far, a gap that the present article aims to fill.
In short, we show that the distance to the TM does not need to be uniformly low for all transition density functions, but that the distance is permitted to scale with the potential energy of the starting state.
The RC received by parametrizing the TM is then of the same quality as in the uniform distance case. This extension to the TM theory will therefore allow to measure the quality of given RCs, and the numerical computation of ideal RCs in systems that been previously deemed unsuitable for the theory.

This paper is structured as follows: 
Section~\ref{sec:good_RCs} reviews the time scale-based definition of good RCs. Section~\ref{sec:weak_reducibility} presents the main contribution of this article, weakened but sufficient conditions for the existence of good RCs. 
In Section~\ref{sec:examples} we give an example of a metastable toy system that fulfills the relaxed but not the original reducibility condition, and demonstrate how the new characterization can improve the quality of error bounds for the dominant timescales. In Section~\ref{sec:outlook}, concluding remarks and an outlook on future work are given.

\section{Good reaction coordinates}
\label{sec:good_RCs}

Before introducing the (generalized) transition manifold framework, we first revisit the fundamental time scale-based definition of good reaction coordinates.

\subsection{Timescales of molecular dynamics}

We consider a \hlC[C5]{time- and space-continuous, reversible and ergodic} Markov process $\mathbf{X}_t$ on a state space $\X\subset\R^n$. \hlC[C5]{In a molecular dynamical system consisting of $N$ atoms, $\X$ often is the Euclidean space describing the three-dimensional positions of all atoms, i.e., $\X=\R^{3N}$ (or $\X=\R^{6N}$ if the atom's momenta are also included). In this case, $\mathbf{X}_t$ is typically} described by a thermostated Hamiltonian dynamics or Langevin dynamics. 

$\mathbf{X}_t$ is fully characterized by its stochastic transition functions $p^t(x,\cdot):\X\rightarrow \R^+$, or, equivalently, by its family of \emph{transfer operators} $\mathcal{T}^t:L^1_\mu \rightarrow L^1_\mu,~t\geq 0$,
$$
\mathcal{T}^t u(x) = \int_\X \frac{\rho(x')}{\rho(x)} p^t(x',x) u(x')\ts dx'.
$$
Here, $\rho$ is the system's (positive) stationary density, which is unique due to the ergodicity of $\mathbf{X}_t$, and $\mu$ is the associated invariant measure.
Operating on $L^1_\mu$, $\mathcal{T}^t$ can be understood as the evolution operator of densities with respect to $\mu$ under the dynamics. 

On $L^1_\mu$, $\mathcal{T}^t$ is a linear Markov operator,~\cite[Chapter~3]{LM13}, and in particular non-expansive. Hence, \hlC[C6]{no eigenvalue of $\mathcal{T}^t$ has absolute value greater than $1$. Due to the uniqueness of the stationary density, the eigenvalue $\lambda_0^t := 1$ is single; the associated unique eigenfunction is $\varphi_0\equiv 1$}.

Furthermore, $\mathcal{T}^t$ is well-defined as an operator $\mathcal{T}^t:L^p_\mu\rightarrow L^p_\mu$ for any $1\leq p \leq \infty$~\cite{BaRo95}. We understand $\mathcal{T}^t$ as an operator on $L^2_\mu$ from now on, where we will be able to exploit the additional Hilbert space structure. In particular, $\mathcal{T}^t$ is self-adjoint with respect to the inner product on $L^2_\mu$~\cite{SchCa92}, hence its point spectrum is real and therefore confined to the interval $(-1,1]$.
\hlC[C7]{Note that $\mathcal{T}^t$ cannot possess the eigenvalue $-1$, as this would imply the existence of an eigenfunction $\widetilde{\varphi}_0\neq \varphi_0$ of $\mathcal{T}^{2t}$ to eigenvalue $1$. This however contradicts the uniqueness of $\varphi_0$ as the only eigenfunction to eigenvalue 1 of $\mathcal{T}^t$ for all $t$.}

In the following we will always order the eigenvalues so that
$$
1=\lambda_0^t >\lambda_1^t\geq \lambda_2^t \geq \cdots~.
$$

The associated eigenfunctions $\varphi_i$ of $\mathcal{T}^t$ form an orthonormal basis of $L^2_\mu$. Hence, on $L^2_\mu$, $\mathcal{T}^t$ admits the decomposition
$$
\mathcal{T}^t = \sum_{i=0}^\infty \lambda_i^t\ts \langle\varphi_i , \cdot \rangle_\mu\ts \varphi_i,
$$
which lets us examine the behavior of $\mathbf{X}_t$ on different time scales. The $i$-th \emph{relaxation rate}, i.e., the \hlC[C8]{exponential rate with} which the $i$-th eigenfunction $\varphi_i$ of $\mathcal{T}^t$ decays, is given by
\begin{equation}
\label{eq:eigenvalues_timescales}
\sigma_i = -\log(\lambda_i^t) / t, \quad i=0,1,2,\ldots,
\end{equation}
independent of $t$. These rates, as well as their inverse, the \emph{relaxation time scales} $t_i = 1/\sigma_i,~i=0,1,2,\ldots$, measure the influence of the different $\varphi_i$ on the long time density transport under $\mathcal{T}^t$, and hence are central quantities of the system.


\subsection{Reaction coordinates}

A reaction coordinate (RC) now is a continuous map $\xi:\X\rightarrow \Y\subset\mathbb{R}^r$, where typically $r\ll n$. Note that the term ``reaction coordinate'' does not imply that $\xi$ describes a reaction of some sort, it simply is a continuous map.
For $y\in \Y$, let $\Sigma_\xi(y)$ be the $y$-level set of $\xi$, i.e.,
$$
\Sigma_\xi(y) = \big\{ x\in\X~\big|~\xi(x)=y\big\}.
$$
Following~\cite{LeLe10}, we now define the \emph{coordinate projection operator} $\Pi_\xi: L^1_\mu\rightarrow L^1_\mu$ for a RC $\xi$ by
\begin{align*}
\big(\Pi_\xi u\big) (x) &= \int_{\Sigma_\xi(\xi(x))} u(x') d\mu_{\xi(x)}(x') \\
&= \frac{1}{\Gamma\big(\xi(x)\big)} \int_{\Sigma_\xi(\xi(x))} u(x') \rho(x') \det\big( \nabla\xi(x')^\intercal \nabla\xi(x')\big)^{-1/2}\ts d\sigma_{\xi(x)}(x'),
\intertext{where~$\Gamma(y)$ is a normalization constant given by}
\Gamma(y) &= \int_{\Sigma_\xi(y)} \rho(x')\det\big(\nabla\xi(x')^\intercal \nabla\xi(x')\big)^{-1/2}\ts d\sigma_y(x'),
\end{align*}
and $d\sigma_y$ denotes the surface measure on $\Sigma_\xi(y)$. $\mu_y$ can be understood as the invariant measure $\mu$ conditioned on $\Sigma_\xi(y)$, and formally is induced by the density 
$$
\rho_y = \frac{\rho}{\Gamma(y)} \big[\det\big(\nabla \xi^\intercal \nabla \xi\big) \big]^{-1/2}.
$$

 As $L^2_\mu\subset L^1_\mu$ due to Hölder's inequality, $\Pi_\xi$ is defined on $L^2_\mu$ as well.
Informally, $\Pi_\xi$ has the effect of averaging an input function $u$ over each level set $\Sigma_\xi(y)$ with respect to $\mu_y$.

It has been shown in~\cite{Bittracher2017} that $\Pi_\xi$ is indeed a projection operator. Moreover, $\Pi_\xi$ is equivalent to the Zwanzig projection operator, described in detail in~\cite{Givon04}, although the latter is typically constructed so that its image are functions over the reduced space $\Y$. For our presentation, however, it is advantageous to define $\Pi_\xi$ to project onto a true subspace of $L^2_\mu$ (namely the subspace of functions that are constant on each $\Sigma_\xi(y),~y\in \Y$).

The \emph{effective transfer operator} $\mathcal{T}^t_\xi:L^2_\mu\rightarrow L^2_\mu$ associated with the RC $\xi$ is now defined by
$$
\mathcal{T}^t_\xi = \Pi_\xi \circ \mathcal{T}^t \circ \Pi_\xi.
$$
Originally considered in~\cite{ZHS16}, $\mathcal{T}^t_\xi$ has been shown to again be self-adjoint and bounded in $L^2_\mu$-norm by $1$~\cite{Bittracher2017}. Hence, the eigenvalues $\lambda^t_{\xi,i},~i=0,1,2,\ldots$ of $\mathcal{T}^t_\xi$ are also confined to the interval $[-1,1]$.

\subsection{Preservation of time scales}

Our characterization of \emph{good} RCs --- originally proposed in~\cite{Bittracher2017} --- now revolves around the central assumption that the relevant part of the dynamics (the ``reaction'') occurs on the slowest time scales of $\mathbf{X}_t$. Moreover, we assume that the time scales of the reaction are well-separated from non-reactive time scales, i.e., $t_0>t_1 \geq \cdots \geq t_d \gg t_{d+1}$ for some $d\in\mathbb{N}$. This is a sensible and commonly made assumption~\cite{Schuette1999,SarichSchuette2012,SarichNoeSchuette2010,variational2013}, as it holds true for many difference classes of chemical and molecular reactions. However, there are relevant molecular systems whose effective behavior cannot be explained by its slowest timescales alone~\cite{MHP17,Wedemeyer02}, and hence valid criticism of the general equivalence of the slow with the relevant time scales exist. Nevertheless, we assume that the reaction in question is associated with the $d$ dominant time scales.

With the goal of preserving the dominant time scales under projection onto the RC, and the close connection between those time scales and the dominant transfer operator eigenvalues~\eqref{eq:eigenvalues_timescales}, we use the following definition of good RCs:

\begin{definition}[Good reaction coordinates,~\cite{Bittracher2017}]
\label{def:goodRC}
Let $\lambda_i^t,~i=0,1,2,\ldots$ and $\lambda_{\xi,i}^t,~i=0,1,2,\ldots$ denote the eigenvalues of $\mathcal{T}^t$ and $\mathcal{T}^t_\xi$, respectively. Let $t_d$ be the last time scale of the system that is relevant to the reaction. Let $\varepsilon>0$.

An RC $\xi:\X\rightarrow \Y$ is called a \emph{$\varepsilon$-good RC}, if for all $t>0$ holds
\begin{equation}
\label{eq:def_good_RC}
|\lambda_i^t - \lambda_{\xi,i}^t| \leq \varepsilon,\quad i=0,1,\ldots,d.
\end{equation}
Informally, we will call $\xi$ a \emph{good RC} if it is $\varepsilon$-good for small $\varepsilon$.
\end{definition}

Alternatively, the following sufficient condition characterizes good RC by the projection error of the dominant eigenfunctions under~$\Pi_\xi$:

\begin{theorem}[\cite{Bittracher2017},~Corollary 3.6]
\label{thm:projection_error_small}
	Let $(\lambda_i^t,\varphi_i),~i=1,2,\ldots$ denote the eigenpairs of $\mathcal{T}^t$. For any given $i$, if 
	$$
	\|\Pi_\xi \varphi_i - \varphi_i\|_{L^2_\mu}\leq \varepsilon,
	$$
	then there is an eigenvalue $\lambda_{\xi,i}^t$ of $\mathcal{T}^t_\xi$ such that
	$$
	\big| \lambda_i^t - \lambda_{\xi,i}^t \big| \leq \frac{\varepsilon}{\sqrt{1-\varepsilon^2}}.
	$$
\end{theorem}

\hlC[C10, C11]{
\begin{remark}
By the above theorem, choosing the $d$ dominant eigenfunctions as the $d$ components of $\xi$ results in a ``perfect'' RC. However, this approach may lead to redundancy if the $\varphi_i,~i=1,\ldots,d$ are strongly correlated and can be parametrized by a common, lower-dimensional $\xi$. For example, a system with $d$ metastable sets along a common, one-dimensional transition pathway would possess $d$ dominant eigenfunctions, but a one-dimensional good RC that parametrizes the transition pathway (see \cite[Section 5.2]{Bittracher2017} for a detailed example).

Using eigenfunctions as	RCs was also promoted by Froyland et al~\cite{FGH14a, FGH14b}, for the special case where the timescale separation stems from a pointwise local separation of the dynamics into a slow and a fast part. Just like for the transition manifold approach presented in Section~\ref{sec:weak_reducibility}, the short-time equilibration of the dynamics again plays an important part, but unlike in our approach it is assumed to take place on certain ``fast fibers'' of state space. The transition manifold framework can therefore be considered a generalization of the approach of Froyland et al.
\end{remark}
}

\section{Weak reducibility of stochastic systems}
\label{sec:weak_reducibility}

Definition~\eqref{eq:def_good_RC} is not constructive, in that it allow one to check the quality of a given RC, but does not indicate how to find a good RC algorithmically. 
To this end, we will now derive a reducibility condition that binds the existence of good RCs to the existence of a certain low-dimensional structure in the space of transition density functions. This structure, called the \emph{transition manifold}, can be interpreted as the backbone of the essential dynamics, can be visualized, and ultimately can be used to numerically compute good RCs.

\subsection{Condition for good reaction coordinates based on transfer operator eigenfunctions}

It was shown in~\cite{Bittracher2017} that if for some functions $\hat{\varphi}_i:\Y\rightarrow \R$ the condition 
\begin{equation}
\label{eq:pointwise_condition}
\|\varphi_i - \hat{\varphi}_i\circ \xi \|_\infty \leq \varepsilon,\quad i=0,1,\ldots,d
\end{equation}
holds, then $\xi$ is a $\frac{\varepsilon}{\sqrt{1-\varepsilon^2}}$-good RC by Theorem~\ref{thm:projection_error_small}. In other words, if the dominant eigenfunctions are pointwise almost constant along the level sets of $\xi$, then $\xi$ is a good RC.

It turns out, however, that condition~\eqref{eq:pointwise_condition} is unnecessarily strong. To be  precise, the pointwise approximation implied by the $\|\cdot\|_\infty$-norm can be replaced by the following weaker condition. This was already observed previously~\cite[Remark~4.3]{Bittracher2017}, but has not been proven formally.

\begin{theorem}
\label{thm:goodRCcondition}
Assume that for an RC $\xi:\X\rightarrow \Y$ and some functions $\hat{\varphi}_i:\Y\rightarrow \R,~i=0,1,\ldots,d$ holds
\begin{equation}
\label{eq:characterization_eigenfunctions}
\int_{\Sigma_\xi(y)}\big| \varphi_i(x') - \hat{\varphi}_i(y) \big|\ts d\mu_{y}(x') \leq  \varepsilon
\end{equation}
for all level sets $\Sigma_\xi(y)$ of $\xi$. Then
$$
\| \Pi_\xi\varphi_i-\varphi_i \|_{L^2_\mu}  \leq 2\varepsilon.
$$
\end{theorem}

\begin{remark}
In words, for a specific value $y\in\Y$, the dominant eigenfunctions $\varphi_i$ do not need to be almost constant everywhere on $\Sigma_\xi(y)$, but only the average deviation of $\varphi_i$ from some value $\hat{\varphi}(y)$ along $\Sigma_\xi(y)$, weighted by $\mu_y$, must be small.
Hence, $\xi$ may be a good RC even if $\varphi_i(x')$ substantially deviates from the value $\hat{\varphi}(y)$, as long as it is in regions where the measure $\mu_y$ is small. These are precisely the regions of state space that are lowly-populated in the canonical ensemble, and thus are statistically irrelevant.
\end{remark}

\begin{proof}[Proof of Theorem~\ref{thm:goodRCcondition}]
The projection error is
$$
\|\Pi_\xi\varphi_i - \varphi_i \|_{L^2_\mu} \leq \|\Pi_\xi\varphi_i - (\hat{\varphi}_i\circ \xi ) \|_{L^2_\mu} + \|(\hat{\varphi}_i\circ \xi ) - \varphi_i \|_{L^2_\mu} .
$$
For the first summand, consider
\begin{align*}
\big(\Pi_\xi\varphi_i\big)(x) &= \int_{\Sigma_\xi(\xi(x))} \varphi_i(x') d\mu_{\xi(x)}(x') \\
&= \int_{\Sigma_\xi(\xi(x))} \Big(\hat{\varphi}_i\big(\underbrace{\xi(x')}_{=\xi(x)}\big) + \varphi_i(x')  - \hat{\varphi}_i \big(\xi(x')\big)\Big) d\mu_{\xi(x)}(x') \\
&= \hat{\varphi}_i\big(\xi(x)\big) + \int_{\Sigma_\xi(\xi(x))} \Big( \varphi_i(x')  - \hat{\varphi}_i \big(\xi(x')\big)\Big) d\mu_{\xi(x)}(x'),
\intertext{and hence}
\|\Pi_\xi\varphi_i - (\hat{\varphi}_i\circ \xi) \|_{L^2_\mu}^2 & \leq\int_\X \Big( \underbrace{ \int_{\Sigma_\xi(\xi(x))} \big| \varphi_i(x') - \hat{\varphi}_i\big(\xi(x')\big) \big| d\mu_{\xi(x)}(x')}_{\leq \varepsilon}\Big)^2 d\mu(x) \\
&\leq \varepsilon^2 \int_\X d\mu(x) = \varepsilon^2.
\end{align*}
For the second summand, we get with the co-area formula~\cite{Evans15}
\begin{align*}
\|(\hat{\varphi}_i\circ \xi ) - \varphi_i \|_{L^2_\mu}^2  &= \int_{\Y} \int_{\Sigma_\xi(y)} \big| \hat{\varphi}_i\big(\xi(x')) - \varphi_i(x') \big|^2 \ts d\mu_y(x') \ts \Gamma(y)\ts dy \\
&\leq \int_{\Y} \Big(\underbrace{\int_{\Sigma_\xi(y)} \big| \hat{\varphi}_i\big(\xi(x')) - \varphi_i(x') \big| \ts d\mu_y(x')}_{\leq\varepsilon}\Big)^2 \ts \Gamma(y)\ts dy \\
&\leq \varepsilon^2 \int_{\Y} \Gamma(y)\ts dy = \varepsilon^2.
\end{align*}
\end{proof}

\subsection{Weak reducibility and weak transition manifolds}

From the abstract condition~\eqref{eq:characterization_eigenfunctions} of good RCs, one can now derive a constructive condition for the existence of a good RC.
We will also repeat the strong version of this condition, based on~\eqref{eq:pointwise_condition}, which was originally derived in~\cite{Bittracher2017}.

The parametrizations of certain manifolds will play a central role in our constructions. 
\hlC[C12]{Specifically, we consider the special class of manifolds $\M\subset L^1$ for which a compact and connected set $\Y\subset \R^r$, as well as a homeomorphism $\mathcal{E}:\M\rightarrow\Y$ exists, such that
\begin{equation}
	\label{eq:manifold_parametrization}
	\M = \mathcal{E}^{-1}(\Y).
\end{equation}
$\Y$ will later become the image space of our constructed RC.
}

For a fixed lag time $\tau>0$, we now call the set of functions
$$
\widetilde{\M} =\big \{ p^\tau(x,\cdot)~|~x\in\X \big\}\subset L^1
$$
the \emph{fuzzy transition manifold}. Note that $\widetilde{\M}$ is not a manifold; the reason behind the choice of name will however soon become clear. \hlC[C12]{Now, for any manifold $\M\subset \widetilde{\M}$ of form~\eqref{eq:manifold_parametrization}}, define the projection onto $\M$ by
\begin{equation}
\label{eq:Qprojection}
\Q:\X \rightarrow \M, \quad x \mapsto \argmin_{f\in\M} \| f - p^\tau(x,\cdot) \|_{L^2_{1/\mu}}.
\end{equation}

\begin{definition}
\label{def:reducibility}
We call the system \emph{strongly $(\varepsilon,r,\tau)$-reducible}, if there exists a manifold $\M\subset \widetilde{\M}$ \hlC[C12]{of form~\eqref{eq:manifold_parametrization}} so that for all $x\in\X$
\begin{equation}
\label{eq:stronly_reducible}
\big\| \Q(x) - p^\tau(x,\cdot) \big\|_{L^2_{1/\mu}} \leq \varepsilon.
\end{equation}
We call any such $\M$ a \emph{strong transition manifold}.

We call the system \emph{weakly $(\varepsilon,r,\tau)$-reducible}, if there exists a manifold $\M\subset \widetilde{\M}$ \hlC[C12]{of form~\eqref{eq:manifold_parametrization}} so that for all $x\in\X$
\begin{equation}
\label{eq:weakly_reducible}
\int_{\Sigma_\Q(\Q(x))} \big\| \Q(x') - p^\tau(x',\cdot) \big\|_{L^2_{1/\mu}} \ts d\mu_{\Q(x)}(x') \leq \varepsilon,
\end{equation}
where $\Sigma_\Q(f)$ is the $f$-level set of $\Q$. We call any such $\M$ a \emph{weak transition manifold}.
\end{definition}

\hlC[C13]{
\begin{example}
As an illustration of the core idea behind the TM construction, we give a simple example of a metastable system with a strong TM, originally published in~\cite{BKHKS19}.

Consider a two-dimensional system described by the overdamped Langevin equation
\begin{equation}
\label{eq:overdampedLangevin}	
d\mathbf{X}_t = -\nabla V(\mathbf{X}_t)\ts dt+ \sqrt{2\beta^{-1}} d\mathbf{W}_t,
\end{equation}
where $V$ is the potential energy function and $\mathbf{W}_t$ is a Wiener diffusion process scaled by the inverse temperature $\beta \in \R^+$. Now suppose that $V$ possesses two local energy wells, connected by a linear, one-dimensional transition path, such as in Figure~\ref{fig:Transition manifold concept}~(left).
The ``reaction'' in this system is the rare transition from one well to the other. Hence, an intuitively good RC is the horizontal coordinate of a point, $\xi(x)=x_1$, as it describes the progress of $x$ along the transition pathway.

The key insight now is that, if the lag time $\tau$ was chosen long enough for a typical trajectory to move to one of the metastable sets, then the transition densities $p^\tau(x,\cdot)\in L^1$ also essentially depend only on the progress of $x$ along the transition path. The reason is that the $p^\tau(x,\cdot)$ are essentially convex combinations of two Gaussians\footnote{To be precise, the $p^\tau(x,\cdot)$ are approximately convex combinations of the quasi-stationary densities~\cite{gesua_jump_2016} of the metastable sets, that here however resemble Gaussians.} centered in the energy minima $A$ and $B$,
$$
p^\tau(x,\cdot)\approx c(x)\rho_A(\cdot) + (1-c(x)) \rho_B(\cdot)
$$
 with the convex factor $c(x)$ determined by the progress of the starting point $x$ along the transition path. This is represented in Figure~\ref{fig:Transition manifold concept}~(right) by the fact that the transition densities for each gray and white starting point, respectively, concentrate around one point each in $L^1$. Hence, overall, the fuzzy TM $\widetilde{\M}$ concentrates around a one-dimensional manifold in $L^1$. This manifold is therefore a strong TM.
 
An example of a system with only a weak TM will be discussed in detail in Section~\ref{sec:examples}.

\begin{figure}
\centering
\includegraphics[width=\textwidth]{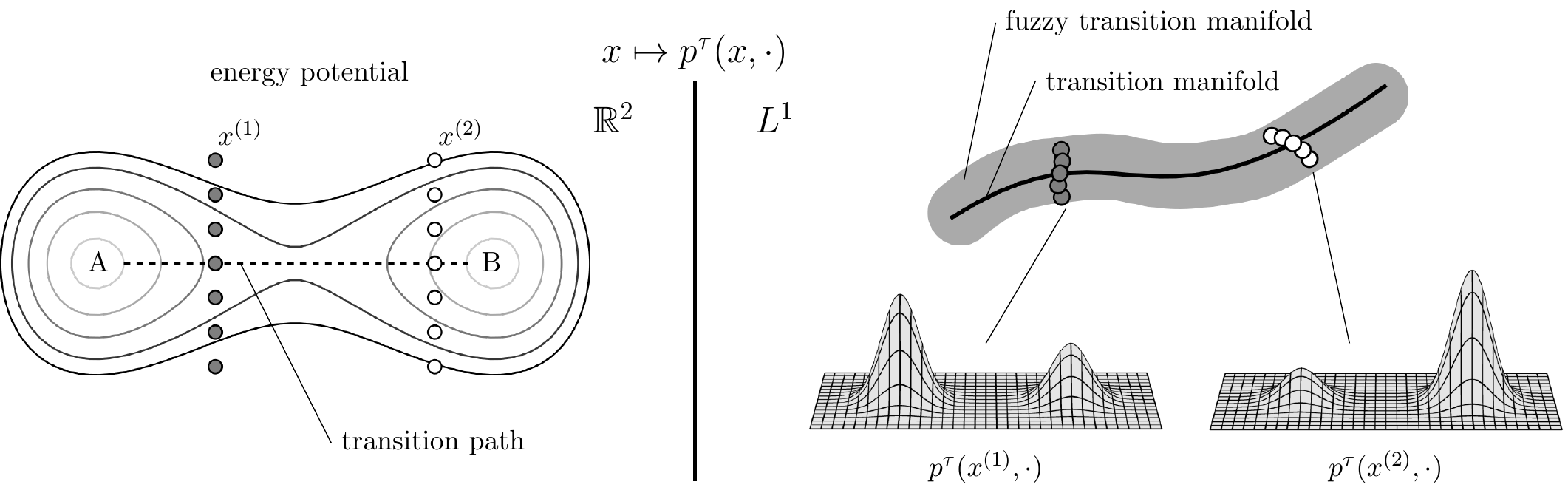}
\caption{\hlC[C13]{Illustration of the transition manifold concept for metastable systems. Left: energy potential of a two-dimensional metastable system. Right: Sketch of the (fuzzy) TM for this system. Starting points $x$ with the same progress along the transition path get mapped to approximately the same density under the map $x\mapsto p^\tau(x,\cdot)$. Geometrically, this means that the fuzzy TM concentrates around a one-dimensional manifold in~$L^1$.}}
\label{fig:Transition manifold concept}
\end{figure}

\end{example}
}

\begin{remark}
Note that we slightly deviate from the original definition of the transition manifold in~\cite{Bittracher2017} by requiring that $\M\subset\widetilde{\M}$ instead of only $\M\subset L^1$.
Also note that $\Q$ is now defined on  $\X$ and not on $\widetilde{\M}$ as originally in~\cite{Bittracher2017}. The interpretation of $\Q$ as ``closest point projection onto $\M$'' is still valid, however.
\end{remark}

Condition~\eqref{eq:stronly_reducible} indicates whether the fuzzy TM $\widetilde{\M}$ clusters $\varepsilon$-closely around an actual manifold $\M$ with respect to the $L^2_{1/\rho}$-norm. Again, condition~\eqref{eq:weakly_reducible} represents a relaxation of this condition, as the integral introduces a weighting with respect to $d\mu_{\Q(x)}$. Informally speaking, for points $x'$ with $\rho(x')=\mathcal{O}(\varepsilon)$, a distance $\big\|\Q(x') - p^\tau(x',\cdot)\big\|_{L^2_{1/\mu}} = \mathcal{O}(1)$ is now permitted without violating the reducibility condition.

It was shown in~\cite{Bittracher2017} that strongly reducible systems possess good RCs. The following theorem now shows that weakly reducible systems still possess good RCs. It characterizes $\Q$ as a good ``$\M$-valued RC'' (cf.~\eqref{eq:characterization_eigenfunctions}):

\begin{theorem}
\label{thm:MvaluedRC}
Let the system be weakly $(\varepsilon,r,\tau)$-reducible. Then for each eigenpair $(\lambda_i^\tau,\varphi_i)$ of the transfer operator $\mathcal{T}^\tau$  there exists a map $\tilde{\varphi}_i:\M\rightarrow \R$ so that for all $x\in \X$
$$
\int_{\Sigma_\Q(\Q(x))} \big| \varphi_i(x') - \tilde{\varphi}_i\big(\Q(x')\big) \big|\ts d\mu_{\Q(x)}(x') \leq \frac{\varepsilon}{|\lambda_i^\tau|}.
$$
\end{theorem}

\begin{proof}
As $\M\subset \widetilde{\M}$, for $x\in\X$ we can choose $q(x)\in \X$ so that $\Q(x) = p^t\big(q(x),\cdot\big)$. Let $\tilde{\varphi}_i:\M\rightarrow \R$ be defined by
$$
\tilde{\varphi}_i\big(\Q(x)\big) = \varphi_i\big( q(x)\big).
$$
Then
\begin{align*}
\int_{\Sigma_\Q(\Q(x))} \big| \varphi_i(x') - \tilde{\varphi}_i\big(\Q(x')\big) \big|\ts d\mu_{\Q(x)}(x') &=
\int_{\Sigma_\Q(\Q(x))} \big| \varphi_i(x') - \tilde{\varphi}_i\big(\Q(x)\big) \big|\ts d\mu_{\Q(x)}(x') \\
&= \int_{\Sigma_\Q(\Q(x))} \big| \varphi_i(x') - \varphi_i\big(q(x)\big) \big|\ts d\mu_{\Q(x)}(x') =: (\star)
\end{align*}

As the system is reversible, the detailed balance condition~$\rho(x)p^\tau(x,x'') = \rho(x'')p^\tau(x'',x)$ holds. Hence, the eigenfunctions $\varphi_i$ of $\mathcal{T}^\tau$ have the property
$$
\lambda_i^\tau \varphi_i = \mathcal{T}^\tau\varphi_i = \int_\X \frac{\rho(x'')}{\rho(x)}p^\tau(x'',\cdot) \varphi_i(x'')~dx'' = \int_\X \varphi_i(x'') p^\tau(\cdot,x'')\ts dx'',
$$
and thus
\begin{align*}
(\star) &= \int_{\Sigma_\Q(\Q(x))}  \frac{1}{|\lambda_i^\tau|} \bigg|\int_\X \varphi_i(x'') \Big( p^\tau(x',x'') - p^\tau\big(q(x),x''\big)\Big)\ts dx''\bigg|\ts d\mu_{\Q(x)}(x').
\intertext{Swapping integrals gives}
(\star)&\leq \frac{1}{|\lambda_i^\tau|} \int_\X \big|\varphi_i(x'')\big|  \int_{\Sigma_\Q(\Q(x))} \Big| p^\tau(x',x'') - p^\tau\big(q(x),x''\big) \Big|\ts d\mu_{\Q(x)}(x')\ts dx'',
\intertext{and with Hölder's inequality, $\|fg\|_{L^1}\leq \|f\|_{L^2_\mu} \|g\|_{L^2_{1/\mu}}$, we get}
&\leq \frac{1}{|\lambda_i^\tau|} \underbrace{\|\varphi_i\|_{L^2_\mu}}_{=1} \ts \bigg\|  \int_{\Sigma_\Q(\Q(x))} \Big| p^\tau(x',\cdot) - p^\tau\big(q(x),\cdot\big) \Big|\ts d\mu_{\Q(x)}(x') \bigg\|_{L^2_{1/\mu}}.
\end{align*}
Applying triangle inequality and using $p^\tau\big(q(x),\cdot\big)=\Q(x)$ gives

\begin{align*}
(\star)&\leq \frac{1}{|\lambda_i^\tau|} \int_{\Sigma_\Q(\Q(x))} \Big\| p^t(x',\cdot) - p^t\big(q(x),\cdot\big) \Big\|_{L^2_{1/\mu}}\ts d\mu_{\Q(x)}(x') \\
&= \frac{1}{|\lambda_i^\tau|} \int_{\Sigma_\Q(\Q(x))} \big\|p^t(x',\cdot) - \underbrace{\Q(x)}_{=\Q(x')}\big\|_{L^2_{1/\mu}}\ts d\mu_{\Q(x)}(x').
\end{align*}
By our assumption, this integral is at most $\varepsilon$. Hence,
$$
(\star) \leq \frac{\varepsilon}{|\lambda_i^\tau|}.
$$
\end{proof}

As the last step, we can now construct from $\Q$ an $r$-dimensional RC that meets the condition~\eqref{eq:def_good_RC}:
\begin{corollary}
Let the system be weakly $(\varepsilon,r,\tau)$-reducible. Let $\mathcal{E}:\M\rightarrow\R^r$ be any parametrization of the transition manifold $\M$. Then for the RC
\begin{equation}
\label{eq:idealRC}
\xi:\X \rightarrow \R^r,\quad x \mapsto \mathcal{E}\big(\Q(x)\big)
\end{equation}
and the eigenpairs $(\lambda_i^\tau,\varphi_i)$ of $\mathcal{T}^\tau$ holds 
\begin{equation}
\label{eq:projection_error_eigenvalue}
\|\Pi_\xi\varphi_i - \varphi_i\|_{L^2_\mu} \leq \frac{2\varepsilon}{|\lambda_i^\tau|}.
\end{equation}
\end{corollary}
\begin{proof}
Let $\tilde{\varphi}_i:\M\rightarrow \R$ as in the proof of Theorem~\ref{thm:MvaluedRC}, and define $\hat{\varphi}_i:\Y\rightarrow \R$ via 
$$
\hat{\varphi}_i(y):=\tilde{\varphi}_i\big(\mathcal{E}^{-1}(y)\big).
$$
Note that for any $x\in \X$ holds $\Sigma_\xi\big(\xi(x)\big) = \Sigma_\Q(\Q(x))$. Thus,
\begin{align*}
\int_{\Sigma_\xi(\xi(x))} \big| \varphi_i(x') - \big(\hat{\varphi}_i \circ \xi\big)(x')\big|\ts d\mu_y(x') &= \int_{\Sigma_\Q(\Q(x))} \big| \varphi_i(x') - \big(\tilde{\varphi}_i \circ \Q\big)(x')\big|\ts d\mu_{\Q(x)}(x') \\
& \leq \frac{\varepsilon}{|\lambda_i^\tau|},
\end{align*}
where the last inequality is Theorem~\ref{thm:MvaluedRC}. The assertion now follows from Theorem~\ref{thm:goodRCcondition}.
\end{proof}
If $(\lambda_i^\tau,\varphi_i)$ is dominant, i.e., $\lambda_i^\tau \approx 1$, then the projection error~\eqref{eq:projection_error_eigenvalue} is small. In that case, $\xi:x\mapsto \mathcal{E}\big(\Q(x)\big)$ is indeed a good RC, by Theorem~\ref{thm:projection_error_small}.

\begin{remark}
Any RC of form~\eqref{eq:idealRC} is called an \emph{ideal RC}~\cite{Bittracher2017}. As in practice, however, neither the projection $\Q$ nor the parametrization $\mathcal{E}$ of $\M$ are known, this RC cannot be computed analytically. Instead, for strongly reducible systems, an approximate parametrization of $\M$ is computed by applying manifold learning methods to a finite sample of the fuzzy TM $\widetilde{\M}$~\cite{Bittracher2017,BBS18,BKHKS19}. Our ongoing efforts to extend these techniques to the newly-identified weak reducibility condition will be discussed in the outlook in Section~\ref{sec:outlook}.

\end{remark}

\section{Numerical example: a weakly reducible system}
\label{sec:examples}

In order to compare the strong and weak reducibility condition, we \hlC[C1]{consider a simple two-dimensional metastable system that possesses a one-dimensional RC.} This system, originally considered in~\cite{LeLe10}, is governed by an overdamped Langevin equation \hlC[C13]{of form~\eqref{eq:overdampedLangevin},} where the potential energy function $V$ is given by
$$
V(x) = (x_1^2-1)^2 + 10\ts (x_1^2+ x_2 - 1)^2.
$$
We choose the inverse temperature $\beta=1$, and consider the system on the domain $\X = [-2,2]\times [\hl{-2},2]$ (though no boundary conditions have been enforced in the following computations). \hlC[C1]{The potential $V$, depicted in Figure~\ref{fig:Bananapot}~(a), possesses two local minima in the states $A=(-1,0)$ and $B=(1,0)$. The reaction in question hence is the transition from the area around one minimum (without loss of generality state $A$) to the other (state $B$). The minimum energy pathway (MEP)~\cite{MaFiVaCi06}, which in the zero temperature limit supports almost all reactive trajectories~\cite{Ren03}, is indicated by the white dashed line.}

\begin{figure}
\centering
	\begin{minipage}{.35\textwidth}
	\centering
		\includegraphics[scale=1]{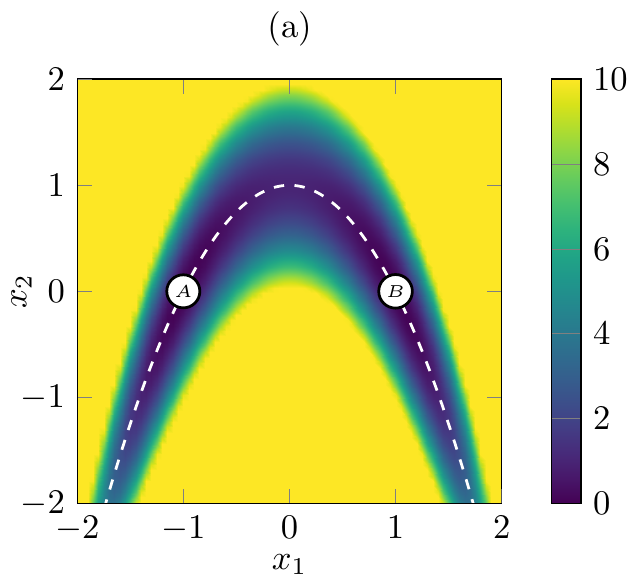}
	\end{minipage}
	\hspace{.1\textwidth}
	\begin{minipage}{.35\textwidth}
		\centering
		\includegraphics[scale=1]{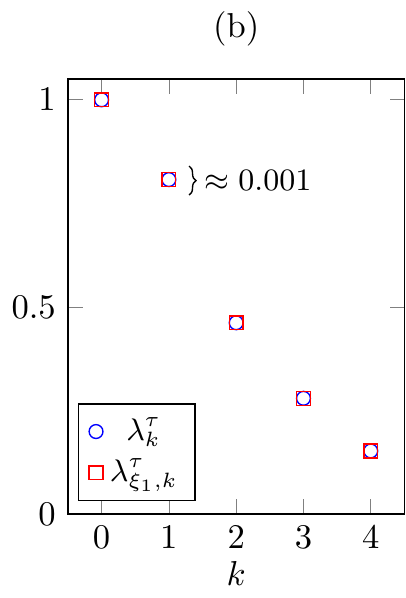}
	\end{minipage}
	\caption{(a) Energy potential of a two-dimensional drift-diffusion system. The reaction of interest here is the transition between the two local minima. \hlC[C1]{(b)~Eigenvalues of the full transfer operator $\mathcal{T}^\tau$ and of the effective transfer operator $\mathcal{T}^\tau_{\xi_1}$ projected onto the computed RC~$\xi_1$.}}
	\label{fig:Bananapot}
\end{figure}

The spectrum of $\mathcal{T}^\tau$ for $\tau = 0.5$, computed by a Ulam method~\cite{Ulam1960} from a long, equilibrated trajectory of the system, exhibits a spectral gap after $\lambda_1$ (Figure~\ref{fig:Bananapot}~(b)). 
The relevant reaction, i.e., the transition between the two metastable sets, is associated primarily with the process on the dominant timescale $t_1$.

The \hlC[C1]{(MEP) of the potential} is given by the set
$$
A_\text{MEP} = \{(x_1,x_2)\in \X~|~x_2 = 1-x_1^2\}.
$$
Intuitively, the manifold 
$$
\M_\text{MEP} = \hlC[C14]{\{ p^\tau(x,\cdot)~|~x\in A_\text{MEP} \}}
$$
should constitute a good TM. This statement should come with a warning: The intuition that the MEP allows to construct a good TM is wrong in general. There are many cases where the relevant transition pathways are completely different from the MEPs of the underlying system, mainly because for finite temperatures all statistically relevant transition paths concentrate in regions not close to the MEP and only converge to the MEP in the limit of zero temperature. In the case considered herein, however, relevant transition paths concentrate around the MEP even for finite temperatures. 

\hlC[C14, C16]{
Before quantitatively assessing whether or not $\M_\text{MEP}$ is indeed is a good TM, we visualize the fuzzy TM \hl{of the system}, i.e., the set $\widetilde{\M}= \{p^\tau(x,\cdot)~|~x\in\X\}$. As $\widetilde{\M}$ lies in the function space $L^1$, it first needs to be embedded into a (finite-dimensional) Euclidean space. This is done by computing the mean of every $p^\tau(x,\cdot)\in\widetilde{\M}$ via the function $\mathbf{m}:L^1\rightarrow \R^2$,
\begin{equation}
\label{eq:meanembedding}
\mathbf{m}\big(p^\tau(x,\cdot)\big) := \int_\X x'\ts p^\tau(x,x')\ts dx'.
\end{equation}
The set $\mathbf{m}\big(\widetilde{\M}\big)$ then serves as the Euclidean embedding\footnote{\hl{While for general dynamics $\mathbf{m}$ is not an embedding of the fuzzy TM in the strict topological sense, we conjecture that in this system, no two transition densities $p^\tau(x_1,\cdot), p^\tau(x_2,\cdot)$ possess the same mean, and hence that $\mathbf{m}$ is homeomorphic on $\widetilde{\M}$ and its image. Still, we neither formally confirm this, nor assess the distortion of $\widetilde{M}$ under $\mathbf{m}$, and hence $\mathbf{m}(\widetilde{\M})$ as a replacement for $\widetilde{M}$ should be handled with care.}} of $\widetilde{\M}$. 

Furthermore, as $\mathbf{m}\big(\widetilde{\M}\big)$ is an infinite set, only a finite subsample can be visualized. For this we draw a large number, specifically $N=8000$, of starting points $\{x_1,\ldots,x_N\}$ uniformly from $\X$ and for each $x_k$ compute $\mathbf{m}_k := \mathbf{m}\big(p^\tau(x_k,\cdot)\big)$.
Here the integral in~\eqref{eq:meanembedding} is approximated via Monte Carlo quadratur, i.e., for $M\gg 1$,
\begin{equation}
\label{eq:monte_carlo_quadrature}
\mathbf{m}\big(p^\tau(x_k,\cdot)\big)  \approx \frac{1}{M}\sum_{l=1}^M z^{(l)}_k,
\end{equation}
where the $z_k^{(l)}$ are samples of the density $p^\tau(x_k,\cdot)$. These were computed numerically by an Euler-Maruyama integrator of~\eqref{eq:overdampedLangevin}, starting in $x_k$, with a different random seed for each $l=1,\ldots,M$.

The points $\mathbf{m}_k$ are shown in Figure~\ref{fig:Bananapot_fuzzyTM}.
}
We observe that most of the $\mathbf{m}_k$ lie close to a parabola-like structure, though there appear to exist systematic outliers, associated with starting points from the high energy regions in the lower part of $\X$. \hlC[C15]{The maximum distance is assumed by the starting point $x^*=(0,-2)$.} The parabola is exactly the Euclidean \hlC[C14]{embedding} of $\M_\text{MEP}$, which is also shown in Figure~\ref{fig:Bananapot_fuzzyTM}.

\begin{figure*}
\centering
	\begin{minipage}{.6\textwidth}
	\centering	
	\includegraphics[scale=1]{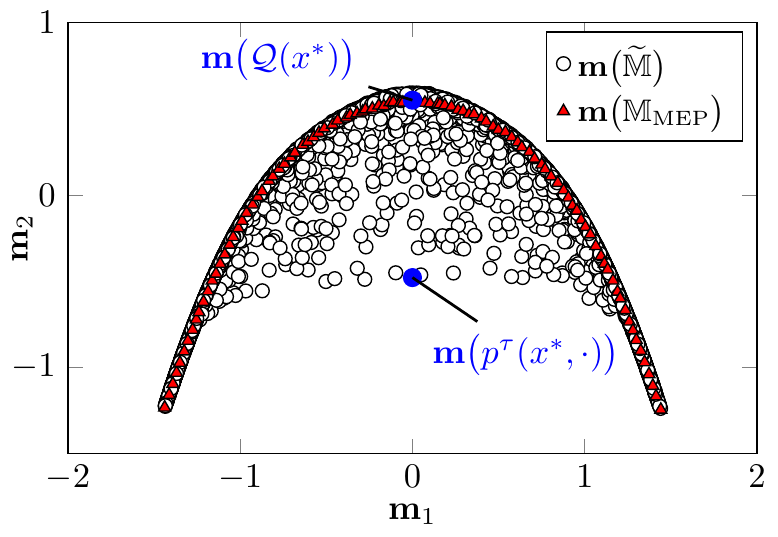}
	\end{minipage}
	\caption{\hlC[C15]{Euclidean embeddings via the mean embedding function $\mathbf{m}$} of the fuzzy TM $\widetilde{\M}$, and the TM based on the minimum energy pathway, $\M_\text{MEP}$. \hlC[C15]{Shown are $N=8000$ sample points of $\mathbf{m}\big(\widetilde{\M}\big)$, and $N=100$ sample points of $\mathbf{m}\big(\M_\text{MEP}\big)$. $\mathbf{m}\big(\widetilde{\M}\big)$ appears to cluster around $\mathbf{m}\big(\M_\text{MEP}\big)$}, except for outliers from the high energy regions below the MEP.}
	\label{fig:Bananapot_fuzzyTM}
\end{figure*}

However, the outliers prevent $\M_\text{MEP}$ from being a good strong TM by Definition~\ref{def:reducibility}. To be precise, for the point $x^*=(0,-3)$, we get for the distance in~\eqref{eq:stronly_reducible}
\begin{equation}
\label{eq:Banana_strongTM_bound}
\big\|\mathcal{Q}(x^*) - p^t(x^*,\cdot)\big\|_{L^2_{1/\mu}} \approx 2.5,
\end{equation}
where again finite samples of $\widetilde{\M}$ and $\M_\text{MEP}$, and kernel density estimations of the $p^t(x,\cdot)$ were used in the computation.
Using~\eqref{eq:Banana_strongTM_bound} as a lower bound for the eigenvalue approximation~\eqref{eq:def_good_RC} via Theorem~\ref{thm:goodRCcondition} and Theorem~\ref{thm:projection_error_small} is of course worthless, hence $\M_\text{MEP}$ is not a strong TM.

On the other hand, for the defining condition~\eqref{eq:weakly_reducible} of weak reducibility holds
\begin{equation}
\label{eq:Banana_weakTM_bound}
\int_{\Sigma_\Q(\Q(x^*))} \big\| \Q(x') - p^\tau(x',\cdot) \big\|_{L^2_{1/\mu}} \ts d\mu_{\Q(x^*)}(x') \approx 0.02
\end{equation}
for the problematic point $x^*$. Assuming this value is indeed an upper bound for~\eqref{eq:weakly_reducible}, the system is weakly reducible with parameter $\varepsilon=0.06$, and $\M_\text{MEP}$ is the corresponding weak TM.
The eigenvalue error for $\lambda_1^\tau$ predicted by Theorem~\ref{thm:goodRCcondition} and Theorem~\ref{thm:projection_error_small} then is
\begin{equation}
\label{eq:predicted_error_bound}
| \lambda_1^\tau - \lambda_{\xi,1}^\tau | \leq 0.06,
\end{equation}
for any RC $\xi$ of the form~\eqref{eq:idealRC}. 

To confirm this error bound, we \hl{now} construct such an RC. For this, \hlC[C17]{a parametrization $\mathcal{E}$ of $\M_\text{MEP}$ must be chosen. Any such parametrization is sufficient, for simplicity we choose 
$$
\mathcal{E}\big(p^\tau(x,\cdot)\big) := x_1,
$$
i.e., the map of $p^\tau(x,\cdot)$ onto the first component $x_1$ of its starting point $x$.}
Next, the projection $\mathcal{Q}$ of $\widetilde{\M}$ onto the TM $\M_\text{MEP}$ is required. In order to avoid the costly calculation of kernel density estimates for the large number of starting points, and to avoid the badly-conditioned scaling by the factor $1/\rho$,
we replace the $L^2_{1/\rho}$ distance in~\eqref{eq:Qprojection} by the Euclidean distance between the mean-embedded densities, i.e., utilize
\hlC[C18]{
$$
\widetilde{\mathcal{Q}}(x) = \argmin_{f\in \M_\text{MEP}} \big\| \mathbf{m}(f) - \mathbf{m}\big(p^\tau(x,\cdot)\big) \big\|_2.
$$
Numerically, this projection is approximated by choosing from the 100 sample points of $\mathbf{m}(\M_\text{MEP})$ that are shown in Figure~\ref{fig:Bananapot_fuzzyTM} the point of minimum distance from $\mathbf{m}(p^\tau(x,\cdot))$. The point $\mathbf{m}(p^\tau(x,\cdot))$ is here again computed via~\eqref{eq:monte_carlo_quadrature}.
}
While using the projection $\widetilde{\Q}$ instead of $\Q$ might slightly distort the computed RC, it will have a negative impact on the quality of the RC, so if the bound~\eqref{eq:predicted_error_bound} holds for $\widetilde{\Q}$, it will hold for $\Q$ as well. Moreover, it has been shown in~\cite{BKHKS19} that the $L^2_{1/\rho}$ distance is equivalent to the distance in certain embedding spaces.

The \hl{final} RC is then given by $\xi_1:x \mapsto \mathcal{E}\big(\widetilde{\mathcal{Q}}(p^\tau(x,\cdot))\big)$. By numerically evaluating $\xi_1$ at the 8000 sample points (where the $p^\tau(x,\cdot)$ are again approximated by finite samples) and interpolating the resulting values bilinearly, we receive a continuous RC on $\X$. Figure~\ref{fig:Bananapot_RCs} shows the level plot of~$\xi_1$.
\hlC[C10]{
We see that the level sets of $\xi_1$ are essentially identical to those of the dominant eigenfunction $\varphi_1$, also shown in Figure~\ref{fig:Bananapot_RCs}. This is not surprising, as~$\xi_1$ is constructed to fulfill the requirements of Theorem~\ref{thm:projection_error_small} , i.e., 
the dominant eigenfunctions are required to be almost invariant under averaging over the level sets of $\xi_1$. As there is only one dominant eigenfunction $\varphi_1$, and $\xi_1$ is also one-dimensional, this implies that the level sets of $\xi_1$ and $\varphi_1$ are almost identical. Note however that the precise ranges of $\xi$ and $\varphi_1$ are not necessarily identical, but strongly depend on the chosen parametrization $\mathcal{E}$.
}

\begin{figure}
\centering
		\includegraphics[scale=1]{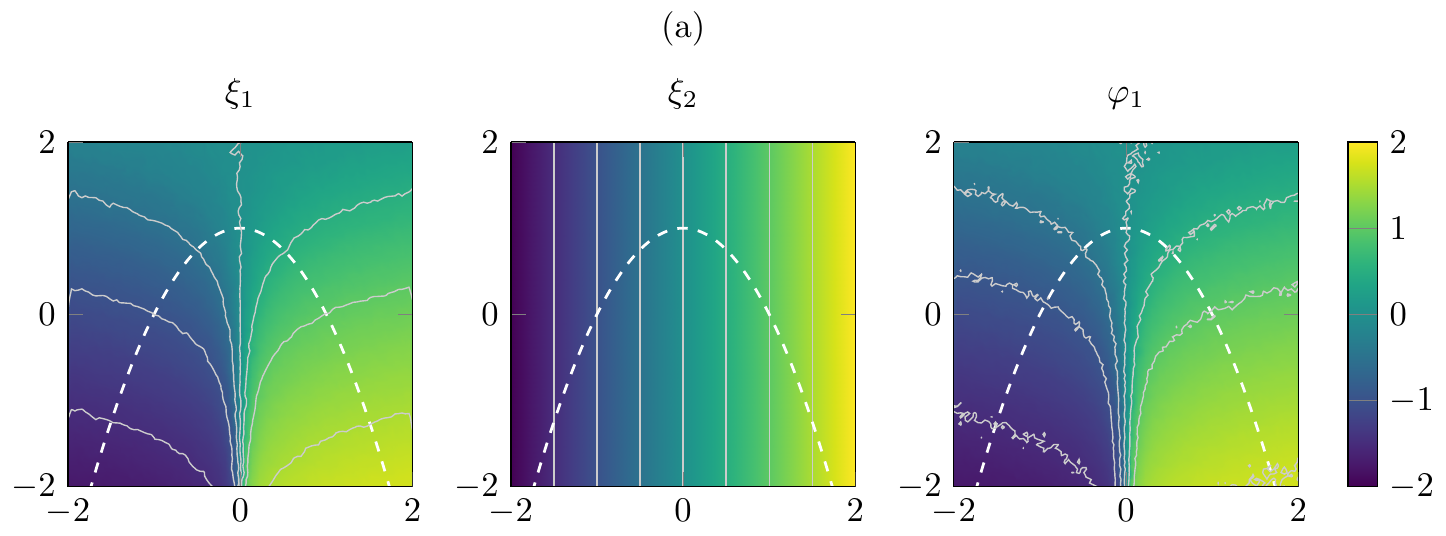}
		\caption{\hlC[C10]{Level plots of the RCs $\xi_1$ computed by the TM method, a naively-constructed RC $\xi_2$, as well as the dominant eigenfunction $\varphi_1$ of $\mathcal{T}^\tau$. We see that the level sets of $\xi_1$ and $\varphi_1$ are essentially identical.}}
	\label{fig:Bananapot_RCs}
\end{figure}

The effective transfer operator $\mathcal{T}^\tau_{\xi_1}$ associated with $\xi_1$ can again be approximated by an Ulam method. Its leading eigenvalues, shown in Figure~\ref{fig:Bananapot_fuzzyTM}~(b), approximate the eigenvalues of the full transfer operator $\mathcal{T}^\tau$ very well. In particular, for the second dominant eigenvalue holds 
$$
|\lambda^\tau_1 - \lambda^\tau_{\xi,1}| \approx 0.001.
$$
\hlC[C1, C3]{As a consequence, the relaxation rate of the projected system $\xi_1(X_t)$, denoted $\sigma_{\xi_1}$ and computed from $\lambda_{\xi,1}$ via~\eqref{eq:eigenvalues_timescales}, also approximate the rate of the full system $\sigma_\text{full}$ very well; we have $\sigma_{\xi_1}\approx 0.43$, $\sigma_{\text{full}}\approx 0.43$ . In contrast, projections onto other, naively chosen RCs, such as 
$$
\xi_2(x):=x_1,
$$
seem to systematically over-estimate the equilibration rate, hence under-estimates the metastability of the system. Specifically, we have $\sigma_{\xi_2}\approx 0.46$. Reduced models built based on $\xi_2$ would therefore run the risk of equilibrating quicker than the full model by artificially increasing the number of transitions.

That said, the difference between $|\sigma_{\xi_1} - \sigma_{\xi_2}| \approx 0.03$ is rather small, so the naive RC $\xi_2$ can already be considered quite good.
The reason is that at low temperatures the dynamics concentrates near the MEP, and here for each level set of $\xi_2$ there exists a level set of $\xi_1$ that is close (in the sense that the minimum pairwise point distance is small), and the RCs are both smooth. Still, the difference is measurable, and this causes the discrepancy.

Overall, this example} confirms that 
\begin{enumerate}
	\item[1)] the RC $\xi_1$ derived from a parametrization of $\M_\text{MEP}$ is good, and
	\item[2)] the error bound~\eqref{eq:predicted_error_bound} derived from the characterization of $\M_\text{MEP}$ as a weak TM is reasonably accurate.
\end{enumerate}

\section{Conclusion and outlook}
\label{sec:outlook}

In this work, we derived an improved and generalized characterization of good reaction coordinates for timescale-separated stochastic processes. 
We built upon a recently developed framework that constructs good RCs from parametrizations of the so-called transition manifold, a potentially low-dimensional manifold in the space of probability densities. We have shown that the criteria on the underlying system to possess such a manifold were overly strict, in the sense that certain systems with demonstrated good reaction coordinates do not possess a transition manifold by the old definition. We thus provided an alternative, relaxed definition of the transition manifold that is applicable to a larger class of systems, while still allowing the construction of good reaction coordinates.

One natural next step would be to implement the novel definition of weak TMs into a data-driven algorithm for the identification of good RCs. Unlike in the toy example from Section~\ref{sec:examples}, the parametrization of the transition manifold (or of a suitable candidate) is not known analytically in practice. Instead, an approximate parametrization is identified by applying a nonlinear manifold learning algorithm to a large sample of $\widetilde{\M}$ (or a suitable embedding thereof)~\cite{BBS18}. Many manifold learning algorithms, such as the diffusion maps algorithm~\cite{CoLa06} can be tuned to ignore outliers, which can be seen as a heuristic way weighing with respect to the invariant measure $\mu$.
A more rigorous approach however would be to directly implement the weighted distance~\eqref{eq:weakly_reducible} into the diffusion maps algorithm. This could be achieved by using the target measure-extension of diffusion maps~\cite{Banisch20}, which at the same time allows one to estimate the in general unknown measure $\mu$ from data.

\section*{Acknowledgements}

This research has been funded by Deutsche Forschungsgemeinschaft (DFG) through grant CRC 1114 ``Scaling Cascades in Complex Systems'', Project B03 ``Multilevel coarse graining of multi-scale problems''.

\bibliographystyle{abbrv}
\bibliography{weakerTM}

\end{document}